\documentclass[11pt,twoside]{amsart}
\usepackage[russian,english]{babel}
\Russian

\usepackage{amssymb}
\usepackage{latexsym}
\usepackage{calligra}

 \usepackage{mathrsfs}
\usepackage{color}

\catcode`\@=11
\def\omathop#1#2#3{\let\temp=#1\def\letter{#2}
  \ifcat#3_ \let\next\@@olim\else\let\next\@olim\fi\next#3}
\def\@olim{\letter\text{-}\!\temp}
\def\@@olim_#1{\mathchoice{
   \setbox0=\hbox{$\displaystyle\letter\text{-}\!\temp\!\text{-}\letter$}
   \setbox2=\hbox{$\displaystyle\temp$}
   \setbox4=\hbox{$\scriptstyle#1$}
   \dimen@=\wd4 \advance\dimen@ by -\wd2 \divide\dimen@ by2
   \def\next{\letter\text{-}\!\temp_{\hbox to 0pt{\hss$\scriptstyle#1$\hss}}
     \hskip\dimen@}
   \ifdim\wd2>\wd4 \def\next{\@olim_{#1}}\fi
   \ifdim\wd4>\wd0 \def\next{\mathop{\llap{$\letter$-}\!\temp}\limits_{#1}}\fi
   \next}
   {\@olim_{#1}}{\@olim_{#1}}{\@olim_{#1}}}

\newcommand{\be}{\begin{equation}}
\newcommand{\ee}{\end{equation}}

\def\phi{\varphi}
\newcommand{\bi}{\begin{itemize}}
\newcommand{\ei}{\end{itemize}}
\newcommand{\bn}{\begin{enumerate}}
\newcommand{\en}{\end{enumerate}}
\newcommand{\hide}[1]{}          

\theoremstyle{plain}
\newtheorem{thm}{Theorem}[section]

\newtheorem{cor}[thm]{Corollary}
\newtheorem{lemma}[thm]{Lemma}

\newtheorem{example}[thm]{Example}

\theoremstyle{definition}
\newtheorem{definition}[thm]{Definition}

\numberwithin{equation}{section}

\begin{document}

\title{The Dodds-Fremlin type theorem for abstract Uryson operators}

\author{Vladimir~ Orlov, Marat Pliev and Dmitry~ Rode}

\address{Voronezh State University\\
Universitetskaya pl., Voronezh, 394006, Russia}

\address{Southern Mathematical Institute of the Russian Academy of Sciences\\
str. Markusa 22,
Vladikavkaz, 362027 Russia}
%
\address{Voronezh State University\\
Universitetskaya pl., Voronezh, 394006, Russia}



\begin{abstract}
We continue the investigation of abstract Uryson operators in vector lattices. Using the recently proved  ``Up-and-down'' theorem for order bounded, orthogonally additive operators \cite{Pl-7}, we consider the domination problem for $AM$-compact abstract Uryson operators. We obtain the Dodds-Fremlin type theorem and prove that for an $AM$-compact,  positive abstract Uryson operator $T$ from a Banach lattice $E$ to a order continuous Banach lattice $F$, an every abstract Uryson operator $S:E\to F$, such that   $0\leq S\leq T$ is  also $AM$-compact.
\end{abstract}

\maketitle


\section{Introduction}

Today the theory of linear and orthogonally additive operators in vector lattices and lattice-normed spaces  is an  active  area   of Functional Analysis \cite{Abas, AMP, DP, F, Get,Gum,Pl-3,Pl-4, Pl-5, Pl-6,PP,PP-1,PW, PW-1}.   The aim of this article  is  to continue  this line of  investigations and to prove  the Dodds-Fremlin type theorem for abstract Uryson operators acting between Banach lattices.
\footnote{ The first   author  was supported by  Russian Science Foundation (14-21-00066),
 the Ministry of Education and Science of
Russia in frameworks of state task for higher education organizations in
science for 2014-2016 (Project number 1.1539.2014/K), the second   author  was supported by  Russian Foundation for Basic  Research, the grant number 15-51-53119}

\section{Preliminaries}

The  goal of this section is to introduce some basic definitions and facts. General information on vector lattices and   the reader can find in the books \cite{Al,Ku}.

Let $E$ be a vector lattice. A net $(x_\alpha)_{\alpha \in \Lambda}$ in $E$ \textit{order converges} to an element $x \in E$ (notation $x_\alpha \stackrel{\rm (o)}{\longrightarrow} x$) if there exists a net $(u_\alpha)_{\alpha \in \Lambda}$ in $E_{+}$ such that $u_\alpha \downarrow 0$ and $|x_\beta - x| \leq u_\beta$ for all $\beta\in \Lambda$. The equality $x=\bigsqcup\limits_{i=1}^nx_i$ means that $x=\sum\limits_{i=1}^nx_i$ and $x_i \bot x_j$ if $i\neq j$.
An element $y$ of  $E$ is called a \textit{fragment} (in another terminology, a \textit{component}) of an element $x \in E$, provided $y \bot (x-y)$. The notation $y \sqsubseteq x$ means that $y$ is a fragment of $x$.  If $E$ is a vector lattice and $e \in E$ then by $\mathcal{F}_e$ we denote the set of all fragments of $e$.

An element $e$ of a vector lattice $E$ is called a \textit{projection element} if the band  generated by $e$ is a projection band. A vector lattice $E$ is said to have the \textit{principal projection property} if every element of $E$ is a projection element. For instance, every Dedekind $\sigma$-complete vector lattice has the principal projection property.

\begin{definition} \label{def:ddmjf0}
Let $E$ be a vector lattice, and let $F$ be a real linear space. An operator $T:E\rightarrow F$ is called \textit{orthogonally additive} if $T(x+y)=T(x)+T(y)$ whenever $x,y\in E$ are disjoint.
\end{definition}

It follows from the definition that $T(0)=0$. It is immediate that the set of all orthogonally additive operators is a real vector space with respect to the natural linear operations.

\begin{definition}
Let $E$ and $F$ be vector lattices. An orthogonally additive operator $T:E\rightarrow F$ is called:
\begin{itemize}
  \item \textit{positive} if $Tx \geq 0$ holds in $F$ for all $x \in E$;
  \item \textit{order bounded} if $T$ maps order bounded sets in $E$ to order bounded sets in $F$.
\end{itemize}
An orthogonally additive, order bounded operator $T:E\rightarrow F$ is called an \textit{abstract Uryson} operator.
This class of operators  was introduced and studied in 1990 by Maz\'{o}n and Segura de Le\'{o}n \cite{Maz-1,Maz-2}, and then extended to lattice-normed spaces by Kusraev and the second named author \cite{Ku-1,Ku-2,Pl-3}.
\end{definition}
For example, any linear  operator $T\in L_{+}(E,F)$ defines a positive abstract Uryson operator by $G (f) = T |f|$ for each $f \in E$.
Observe that if $T: E \to F$ is a positive orthogonally additive operator and $x \in E$ is such that $T(x) \neq 0$ then $T(-x) \neq - T(x)$, because otherwise both $T(x) \geq 0$ and $T(-x) \geq 0$ imply $T(x) = 0$. So, the above notion of positivity is far from the usual positivity of a linear operator: the only linear operator which is positive in the above sense is zero. A positive orthogonally additive operator need not be order bounded. Consider, for example, the real function $T: \mathbb R \to \mathbb R$ defined by
$$
T(x)=\begin{cases} \frac{1}{x^{2}}\,\,\,\text{if $x\neq
0$}\\
0\,\,\,\,\,\text{if $x=0$ }.\\
\end{cases}
$$

The set of all abstract Uryson operators from $E$ to $F$ we denote by $\mathcal{U}(E,F)$.
Consider some examples.
The most famous one is the nonlinear integral Uryson operator.

\begin{example}
Let $(A,\Sigma,\mu)$ and $(B,\Xi,\nu)$ be $\sigma$-finite complete measure spaces, and let $(A\times B,\mu\times\nu)$ denote the completion of their product measure space. Let $K:A\times B\times\Bbb{R}\rightarrow\Bbb{R}$ be a function satisfying the following conditions\footnote{$(C_{1})$ and $(C_{2})$ are called the Carath\'{e}odory conditions}:
\begin{enumerate}
  \item[$(C_{0})$] $K(s,t,0)=0$ for $\mu\times\nu$-almost all $(s,t)\in A\times B$;
  \item[$(C_{1})$] $K(\cdot,\cdot,r)$ is $\mu\times\nu$-measurable for all $r\in\Bbb{R}$;
  \item[$(C_{2})$] $K(s,t,\cdot)$ is continuous on $\Bbb{R}$ for $\mu\times\nu$-almost all $(s,t)\in A\times B$.
\end{enumerate}
Given $f\in L_{0}(B,\Xi,\nu)$, the function $|K(s,\cdot,f(\cdot))|$ is $\nu$-measurable  for $\mu$-almost all $s\in A$ and $h_{f}(s):=\int_{B}|K(s,t,f(t))|\,d\nu(t)$ is a well defined and $\mu$-measurable function. Since the function $h_{f}$ can be infinite on a set of positive measure, we define
$$
\text{Dom}_{B}(K):=\{f\in L_{0}(\nu):\,h_{f}\in L_{0}(\mu)\}.
$$
Then we define an operator $T:\text{Dom}_{B}(K)\rightarrow L_{0}(\mu)$ by setting
$$
(Tf)(s):=\int_{B}K(s,t,f(t))\,d\nu(t)\,\,\,\,\mu-\text{a.e.}\,\,\,\,(\star)
$$
Let $E$ and $F$ be order ideals in $L_{0}(\nu)$ and $L_{0}(\mu)$ respectively, $K$ a function satisfying $(C_{0})$-$(C_{2})$. Then $(\star)$ defines an \textit{orthogonally additive
 integral operator} acting from $E$ to $F$ if $E\subseteq \text{Dom}_{B}(K)$ and $T(E)\subseteq F$.
\end{example}

Consider the following order in $\mathcal{U}(E,F):S\leq T$ whenever $T-S$ is a positive operator. Then $\mathcal{U}(E,F)$
becomes an ordered vector space.  If a vector lattice $F$ is Dedekind complete we have the following theorem.
\begin{thm}(\cite{Maz-1},Theorem~3.2)\label{th-1}.
Let $E$ and $F$ be a vector lattices, $F$ Dedekind complete. Then $\mathcal{U}(E,F)$ is a Dedekind complete vector lattice. Moreover for $S,T\in \mathcal{U}(E,F)$ and for $f\in E$ following hold
\begin{enumerate}
\item~$(T\vee S)(f):=\sup\{Tg_{1}+Sg_{2}:\,f=g_{1}\sqcup g_{2}\}$.
\item~$(T\wedge S)(f):=\inf\{Tg_{1}+Sg_{2}:\,f=g_{1}\sqcup g_{2}\}.$
\item~$(T)^{+}(f):=\sup\{Tg:\,g\sqsubseteq f\}$.
\item~$(T)^{-}(f):=-\inf\{Tg:\,g;\,\,g\sqsubseteq f\}$.
\item~$|Tf|\leq|T|(f)$.
\end{enumerate}
\end{thm}

Let $E,F$ be vector lattices with $F$ Dedekind complete and $T\in\mathcal{U}_{+}(E,F)$.
By definition
$$
\mathcal{F}_{T}=\{S\in\mathcal{U}_{+}(E,F):\,S\wedge(T-S)=0\}.
$$

For a subset $\mathcal{A}$ of a vector lattice $W$ we employ  the
following notation:
\begin{gather}
\mathcal{A}^{\upharpoonleft}=\{x\in W:\exists\,\,\text{a sequence}\,\,(x_{n})\subset\mathcal{A}\,\,\text{with}\,\,x_{n}\uparrow x\};\notag \\
\mathcal{A}^{\uparrow}=\{x\in W:\exists\,\,\text{a net}\,\,(x_{\alpha})\subset\mathcal{A}\,\,\text{with}\,\,x_{\alpha}\uparrow x\}.\notag
\end{gather}
The meanings of $\mathcal{A}^{\downharpoonleft}$ and $\mathcal{A}^{\downarrow}$ are analogous. As usual, we also write
$$
\mathcal{A}^{\downarrow\uparrow}=(\mathcal{A}^{\downarrow})^{\uparrow};
\mathcal{A}^{\upharpoonleft\downarrow\uparrow}=((\mathcal{A}^{\upharpoonleft})^{\downarrow})^{\uparrow}.
$$
It is clear that $\mathcal{A}^{\downarrow\downarrow}=\mathcal{A}^{\downarrow}$, $\mathcal{A}^{\uparrow\uparrow}=\mathcal{A}^{\uparrow}$.
Consider a positive abstract Uryson operator $T:E\rightarrow F$, where $F$ is Dedekind complete.
Since $\mathcal{F}_{T}$ is a Boolean algebra, it is closed under finite suprema and
infima. In particular, all ``ups and downs'' of $\mathcal{F}_{T}$  are likewise closed under finite suprema and infima, and therefore they
are also directed upward and, respectively, downward.

\begin{definition}\label{def:adm}
A subset $D$ of a vector lattice $E$ is called a \it{lateral ideal} if the following conditions hold
\begin{enumerate}
\item~if $x\in D$ then $y\in D$ for every $y\in\mathcal{F}_{x}$;
\item~if $x,y\in D$, $x\bot y$ then $x+y\in D$.
\end{enumerate}
\end{definition}

Consider some examples.

\begin{example}\label{adm-1}
Let $E$ be a vector lattice. Every order ideal in $E$ is a lateral ideal.
\end{example}

\begin{example}\label{adm-11}
Let $E,F$ be a vector lattices and $T\in\mathcal{U}_{+}(E,F)$. Then $\mathcal{N}_{T}:=\{e\in E:\,T(e)=0\}$  is a lateral ideal.
\end{example}

The following example is important for further considerations.

\begin{lemma}(\cite{Ben}, Lemma~3.5). \label{Ex1}
Let $E$ be a vector lattice and $x\in E$. Then $\mathcal{F}_{x}$ is a lateral ideal.
\end{lemma}\label{le:02}

Let $T\in\mathcal{U}_{+}(E,F)$ and $D\subset E$ be a lateral ideal. Then for every $x\in E$, we  define a map  $\pi^{D}T:E\rightarrow F_{+}$ by the following formula
\begin{gather}
\pi^{D}T(x)=\sup\{Ty:\,y\in\mathcal{F}_{x}\cap D\}.
\end{gather}

\begin{lemma}(\cite{Ben},Lemma~3.6). \label{le:01}
Let $E,F$ be vector lattices  with $F$ Dedekind complete, $\rho\in\mathfrak{B}(F)$, $T\in\mathcal{U}_{+}(E,F)$ and $D$ be a lateral ideal.
Then $\pi^{D}T$ is a positive abstract Uryson operator and $\rho\pi^{D}T\in\mathcal{F}_{T}$.
\end{lemma}

If $D=\mathcal{F}_{x}$ then  the
operator $\pi^{D}T$ is denoted by $\pi^{x}T$. Let $F$ be a vector lattice.
Recall that a family of mutually disjoint order projections  $(\rho_{\xi})_{\xi\in\Xi}$ on $F$ is said to be {\it partition of unity} if $\bigvee\limits_{\xi\in\Xi}(\rho_{\xi})_{\xi\in\Xi}=Id_{F}$.
Any fragment of the form $\sum_{i=1}^{n}\limits\rho_{i}\pi^{x_{i}}T$, $n\in\Bbb{N}$, where $\rho_{1},\dots,\rho_{n}$ is a  finite family of mutually disjoint order projections in $F$, like in  the linear case is
called an {\it elementary} fragment  $T$. The set of all elementary fragments of $T$ we denote by $\mathcal{A}_{T}$.

The following theorem (\cite{Pl-7}, Theor.~3.14) is described the structure of Boolean algebras of fragments of a positive abstract Uryson operator.

\begin{thm}(\cite{Pl-7}, Theor.~3.14).\label{frag-2}
Let $E,F$ be vector lattices, $F$ Dedekind complete, $T\in\mathcal{U}_{+}(E,F)$ and $S\in\mathcal{F}_{T}$.
Then $S\in\mathcal{A}_{T}^{\uparrow\downarrow\uparrow}$.
\end{thm}

\section{Result}
\label{sec5}

In this section we consider a domination problem for $AM$-compact abstract Uryson operators. In the classical sense, the domination problem can be stated as follows. Let $E$, $F$ be vector lattices, $S,T: E \to F$ linear operators with $0 \leq S \leq T$. Let $\mathcal P$ be some property of linear operators $R: E \to F$, so that $\mathcal P(R)$ means that $R$ possesses $\mathcal P$. Does $\mathcal P(T)$ imply $\mathcal P(S)$?

Let $E$ be a vector lattice and $x\in E_{+}$. The order ideal generated by $x$ we denote by $E_{x}$. The following theorem is an important tool for further considerations.  An  {\it $x$-step function}
is any vector $s\in E$ for which there exist pairwise disjoint fragments $x_{1},\dots,x_{n}$ of $x$ with $x=\bigsqcup\limits_{i=1}^{n}x_{i}$ and real numbers $\lambda_{1},\dots,\lambda_{n}$ satisfying $s=\sum\limits_{i=1}^{n}\lambda_{i}x_{i}$.

\begin{thm}(Freudenthal Spectral Theorem)(\cite{Al}, Theorem~2.8).\label{Fr}
Let $E$ be a vector lattice with the principal projection property
and let $x\in E_{+}$. Then for every $y\in E_{x}$
there exists a sequence $(u_{n})$ of $x$-step functions satisfying
$0\leq y-u_{n}\leq\frac{1}{n}x$ for each $n$ and $u_{n}\uparrow y$.
\end{thm}

\begin{definition}
Let $E$ be a vector lattice and $F$ a Banach space. An orthogonally additive operator $T:E\to F$ is called:
\begin{enumerate}
    \item \textit{AM-compact} if for every order bounded set $M \subset E$ its image $T(M)$ is a relatively compact set in $F$;
  \item \textit{C-compact} if the set  $T(\mathcal{F}_{x})$ is relatively compact in $F$ for every $x\in E$.
\end{enumerate}
\end{definition}

\begin{thm} (\cite{Maz-2},Theor.~3.5)\label{thm:Seg-1}
Let $E,F$ be a Banach function spaces, $F$ having order continuous norm. Then every   integral Uryson operator  $T\in\mathcal{U}(E,F)$ is $AM$-compact.
\end{thm}

The next  theorem is  the  main result of the article.
\begin{thm} \label{thm:dom}
Let $E,F$ be  Banach  lattices,   $F$ having  order continuous norm, and  $T\in\mathcal{U}_{+}(E,F)$ be an $AM$-compact operator. Then every   operator $S\in\mathcal{U}_{+}(E,F)$, such that $0\leq S\leq T$ is $AM$-compact.
\end{thm}

Recall that a family of mutually disjoint order projections  $(\rho_{\xi})_{\xi\in\Xi}$ on $F$ is said to be {\it partition of unity} if $\bigvee\limits_{\xi\in\Xi}(\rho_{\xi})_{\xi\in\Xi}=Id_{F}$.
For the proof we need an some auxiliary result.

\begin{lemma} \label{l-1}
Let $E,F$ be vector lattices with    $F$   Dedekind complete, $M\subset E$ be an order bounded set,  $(T_{\alpha})_{\alpha\in\Lambda}\subset\mathcal{U}_{+}(E,F)$ be a decreasing net of positive operators such that $\inf{(T_{\alpha})_{\alpha\in\Lambda}}=0$ and $v_{\alpha}=\sup\{T_{\alpha}(M)\}$. Then $(v_{\alpha})$ is the decreasing net of positive elements of $F$ and  $\inf{(v_{\alpha})_{\alpha\in\Lambda}}=0$.
\end{lemma}
\begin{proof}
Let us show  that  $(v_{\alpha})_{\alpha\in\Lambda}$ is the decreasing net in $F$. Indeed,  $F$ with a vector sublattice of the Dedekind complete vector lattice $C_{\infty}(Q)$ of all extended real valued continuous functions with pointwise ordering in some extremally disconnected compact space $Q$ (more exactly with its image under some vector lattice isomorphism),  (see \cite{Ku}, Theorem 1.4.5).  Take an arbitrary $v_{\alpha}\in F_{+}$, then there exists the partition of unity $(\rho_{\xi})_{\xi\in M}$, such that $\rho_{\xi}v_{\alpha}=\rho_{\xi}T_{\alpha}x_{\xi}$, where $x_{\xi}\in M$ and $M=\{x_{\xi}:\, \xi\in M\}$. Since the net $(T_{\alpha})_{\alpha\in\Lambda}$ is decreasing, we have $\rho_{\xi}T_{\beta}x_{\xi}\leq\rho_{\xi}T_{\alpha}x_{\xi}$ for every $\alpha\leq \beta$, $\alpha,\,\beta\in\Lambda$ and every $\xi\in M$ and therefore $v_{\beta}\leq v_{\alpha}$.  The same arguments show that $\inf{(v_{\alpha})_{\alpha\in\Lambda}}=0$.
\end{proof}

\begin{proof}[Proof of Theorem~\ref{thm:dom}]
Let $T\in\mathcal{U}_{+}(E,F)$ be an $AM$-compact operator, and $x\in E$. Firstly, we prove that operator $\pi^{x}T$ is also $AM$-compact. Fix an order bounded set  $M\subset E$. By definition of the operator $\pi^{x}T$ we have that
$\pi^{x}T(M)\subset T(\mathcal{F}_{x})$. Since, $\mathcal{F}_{x}$ is an order bounded set and operator $T$ is $AM$-compact, we have that $\pi^{x}T(M)$ is the relatively compact subset of $F$ and therefore $\pi^{x}T$ is an $AM$-compact operator. It is clear that  $\rho\pi^{x}T$ is also $AM$-compact, for an every order  projection  $\rho$ on $F$. The finite sum of $AM$-compact operators is also $AM$-compact operator. Thus, the assertion is proved for every $S\in \mathcal{A}_{T}$. Now, let $M$ be an order bounded set in $E$ and $S\in \mathcal{A}_{T}^{\uparrow}$. Then by the definition, there exists a net $(S_{\alpha})\subset \mathcal{A}_{T}$ and a decreasing net $(G_{\alpha})_{\alpha\in\Lambda}\subset \mathcal{U}_{+}(E,F)$, such that
$\inf{(G_{\alpha})_{\alpha\in\Lambda}}=0$ and $|S-S_{\alpha}|\leq G_{\alpha}$, $\alpha\in\Lambda$. Let $v_{\alpha}:=\sup\{G(x):\,x\in M\}$. Then by the Lemma~\ref{l-1} we have that $(v_{\alpha})_{\alpha\in\Lambda}$ is a decreasing net and $\inf(v_{\alpha})_{\alpha\in\Lambda}=0$. Fix $\varepsilon>0$. Since, the norm in $F$ is order continuous there exists a $\alpha_{0}\in\Lambda$, such that $\|v_{\alpha}\|<\frac{\varepsilon}{3}$, for every $\alpha\geq \alpha_{0}$. Let $S_{\alpha}(x_{1}),\dots, S_{\alpha}(x_{n})$ be a finite $\frac{\varepsilon}{3}$-net in $S_{\alpha}(M)$. Then we may write
\begin{align*}
\|S(x)-S(x_{i})\|\leq\|S(x)-S_{\alpha}(x)+S_{\alpha}(x)-S_{\alpha}(x_{i})+S_{\alpha}(x_{i})-S(x_{i})\|\leq\\
\|S(x)-S_{\alpha}(x)\|+\|S_{\alpha}(x)-S_{\alpha}(x_{i})\|+\|S(x_{i})-S_{\alpha}(x_{i})\|\leq\\
\|G_{\alpha}(x)\|+\frac{\varepsilon}{3}+\|G_{\alpha}(x)\|\leq\\
\|v_{\alpha}\|+\frac{\varepsilon}{3}+\|v_{\alpha}\|< \frac{\varepsilon}{3}+
\frac{\varepsilon}{3}+\frac{\varepsilon}{3}=\varepsilon,
\end{align*}
where $x\in M$. Thus $S(x_{1}),\dots, S(x_{n})$ is be a finite $\varepsilon$-net in $S(M)$. The same arguments are valid for every operator $S\in\mathcal{A}_{T}^{\uparrow\downarrow\uparrow}$. By the Theorem~\ref{Fr} an every abstract Uryson operator $S$ such that $0\leq S\leq T$ is  the relatively uniform limit  of some net of elements in the linear span of $\mathcal{A}_{T}^{\uparrow\downarrow\uparrow}$ and therefore is the $AM$-compact operator.
\end{proof}
\begin{cor}
Let $E,F$ be a Banach function spaces  lattice, $F$ having order continuous norm   $F$ be a order continuous Banach lattice and $T\in\mathcal{U}_{+}(E,F)$ be an integral Uryson operator. Then every   abstract  Uryson operator  $S\in\mathcal{U}(E,F)$, such that $0\leq S\leq T$ is $AM$-compact.
\end{cor}
Remark that for linear positive operators the similar theorem  was  proved by  Dodds and Fremlin in \cite{DF}.

\end{document}